\documentclass[12pt]{article}
\usepackage[margin=2cm]{geometry}
\usepackage{amsmath,amsthm,amsfonts,amssymb}
\usepackage{enumitem}
\usepackage[mathscr]{euscript}
\usepackage{amsmath}
\usepackage{array,float}

\usepackage{graphicx,latexsym}
\usepackage{lscape}
\usepackage{multicol}
\usepackage{multirow}
\usepackage{mathrsfs}
\usepackage{graphicx}
\usepackage{lipsum}
\usepackage{subfig}
\usepackage{lscape}
\usepackage{xcolor}
\usepackage{autobreak}
\allowdisplaybreaks
\usepackage{cite}

\newtheorem{thm}{Theorem}

\newtheorem{cor}{Corollary}

\newtheorem{discu}{Discussion:}

\newtheorem{conje}{Conjecture:}

\begin{document}
	\date{}
	\title{\textbf{Wiener and Average Distance of Irregular Square-Cell Configuration}}
	\author{\begin{tabular}{ccc}
				\textbf{S. Prabhu$^{\text 1}$, Sandi Klav\v{z}ar$^{\text{2}}$, M. Anitha$^{\text 1}$,}\\ \textbf{ M. Arulperumjothi$^{\text 3}$, Paul Manuel$^{\text 4}$ }
		\end{tabular}\\
		\begin{tabular}{c}
\small $^{\text 1}$Department of Mathematics, Rajalakshmi Engineering College, Thandalam, Chennai 602105, India \\
\small $^{\text{2}}$Faculty of Mathematics and Physics, University of Ljubljana, Slovenia\\
\small $^{\text 3}$Department of  Mathematics, St. Joseph's College of Engineering, Chennai 600119, India\\ 
\small$^{\text 4}$Department of Information Science, College of Life Sciences, Kuwait University, Kuwait\\
\small *Corresponding author: drsavariprabhu@gmail.com
\end{tabular}}
\maketitle
\vspace{-0.5 cm}
\begin{abstract}
			A subgraph of the square lattice with all of its inner faces being 4-cycles is called a square-cell configuration. Prior work has provided explicit expressions for the total and average distances between vertex pairs in symmetric square-cell configurations, including well-structured families such as hexagonal square-cell configurations $H(n)$, trapezium square-cell configurations $T(n,k)$, and bitrapezium square-cell configurations $BT(n,k_1,k_2)$. In this article, we further extend the square-cell configuration from regular boundaries to irregular boundaries, which do not exhibit complete regularity or symmetry in their structure. We find the generalized expressions for the Wiener index and average distance of such irregular configurations, incorporating combinatorial and structural variations. Our results demonstrate how irregularity affects the growth and distribution of pairwise distances and provide a unifying framework that includes both symmetric and asymmetric square-cell graphs as exceptional cases. This generalization provides novel insights into the structural behaviour of square-cell frameworks characterized by complex or perturbed geometries.
		\end{abstract}
	 average distance 

	
\section{Introduction}

Square-cell configurations are subgraphs of the infinite square lattice that, when embedded in the plane, have the defining property that all inner faces form 4-cycles \cite{GuKlRa00}. These configurations serve as fundamental models in discrete geometry, mathematical chemistry, and theoretical computer science, where they are used to represent molecular grids, communication networks, and planar tilings consisting of vertices and edges (respectively denoted by $V(G)$ and $E(G)$ with respect to the graph $G$). Prior research has extensively analyzed symmetric square-cell configurations such as rectangular blocks, hexagonal grids, trapezium shapes, and bitrapezium forms, yielding closed-form expressions for structural metrics such as the Wiener index and average distance. 

The Wiener index $W(G)$ of a graph $G$ is a well-known topological index (a distance-based index) introduced by Harold Wiener in 1947 to study boiling points of alkanes \cite{Wi47}. The distance (length of a graph geodesic) between the pair of nodes in $G$ is the minimum number of edges among all paths that connect the pair. The sum of the length of a graph geodesic between all pairs (unordered) of nodes in $G$ is known as the Wiener index:
$$W(G) = \frac{1}{2} \sum_{u, v \in V(G)} d_G(u, v)\,,$$
where $d_G(u,v)$ (or $d(u,v)$) is the distance from $u$ to $v$ in $G$. The average distance $\mu(G)$ is obtained by normalizing the Wiener index over the total number of vertex pairs:
\[
\mu(G) = \frac{2W(G)}{n(n-1)},
\]
where $n = |V(G)|$ is the order of the graph \cite{Da93,DoGr77}. As usual, $\delta(G)$ and $\Delta(G)$ respectively indicate the minimum and the maximum degree in a graph $G$.

Much of the existing literature has focused on highly regular square-cell structures, for which the symmetry enables elegant and closed-form distance expressions. For example, the Wiener index for rectangular grids, hexagonal square-cell configuration $H(n)$, trapezium configuration $T(n,k)$, and bitrapezium configuration $BT(n,k_1,k_2)$ have been derived using combinatorial and algebraic techniques. These results have applications in chemical graph theory, particularly in modeling lattice-like molecular compounds such as benzenoid hydrocarbons. However, many real-world networks and molecular structures do not conform to such idealized symmetry. Deformations, edge removals, and growth constraints often yield configurations with holes, jagged boundaries, and irregular connectivity patterns. The symmetry of the square-cell configurations motivates us to study irregular square-cell configuration, which are connected subgraphs of the square lattice and have four cycles for each inner face. Still, the overall layout is asymmetric, non-rectangular, or incomplete. In this article, we extend previous theoretical frameworks by analyzing distance-based indices in an irregular square-cell configuration. We formally define this class of configurations, characterized by topological variations while maintaining the 4-cyclic inner face structure, and derive generalized expressions for $W(G)$ and $\mu(G)$ that account for asymmetries, holes, and vertex degree variations.

Furthermore, we examine how structural irregularities influence the growth and distribution of distances, providing comparative analyses with symmetric square-cell models, and propose a unifying framework that incorporates both symmetric and irregular square-cell graphs as exceptional cases. In chemical graph theory, square, hexagonal and triangular-like configurations were constructed and studied vigorously to analyze their topological properties, such as atom-bond connectivity, molecular properties in terms of their structures, and dual nature properties in their structural activities. Convex triangular networks are a well-studied triangular-cell configuration that has been extended to irregular convex triangular networks. These networks were used in finding their distance-based parameters by uniquely identifying nodes, which has applications in chemistry, network discovery, and pattern verification \cite{PrJeAr24}. 

Similarly, convex benzenoid systems are one of the hexagonal-cell configurations studied for analyzing aromatic compounds consisting of benzene-like rings, which are efficient in increasing the aroma of chemical compounds. These networks are extended to irregular convex benzenoid systems for computing certain counting polynomials \cite{GaMuPr21} and distance and degree-based indices \cite{ArClBa16}, which have broad applications in the study of fluid chemistry, viscosity, molecular physics and quantitative structural studies. In a few articles, the convex benzenoid systems resemble the structure of the honeycomb network and specific graph parameters are derived along with their applications. The readers could refer to Figure~\ref{fig:1} for  examples of grid, honeycomb and convex triangular networks. 

\begin{figure}[ht!]
	\centering 	
	\subfloat[]{\includegraphics[scale=1]{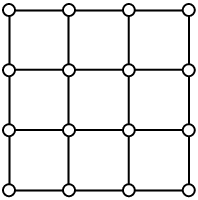}}
	\quad \quad 
	\subfloat[]{\includegraphics[scale=1]{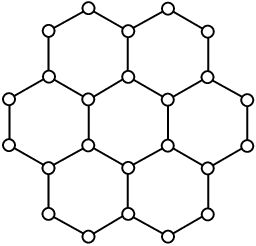}}
	\quad \quad 
	\subfloat[]{\includegraphics[scale=1]{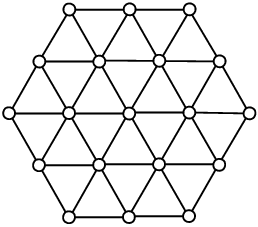}}
	\caption{ (a) Grid network; (b) Honeycomb network; (c) Convex triangular network}
	\label{fig:1}
\end{figure}

This study motivates us to choose the square-cell configuration and extend its structural properties to irregular shapes for computing distance-based parameters, such as the Wiener index and average distances. Previous works on average distances in a square-cell configuration computed by Klav\v{z}ar et al. \cite{GuKlRa00} contain some typographical inaccuracies, most notably that the average distance of the hexagonal square-cell configuration was wrongly derived. In the present work, we first correct these inaccuracies and then extend the study by computing the average distance in an irregular square-cell configuration. 

The article comprises six major sections. Section \ref{s1} promotes the concept of average distance along with its literature. Section \ref{s2} discusses the idea of the Wiener index and the average distance via the cut method, and Section \ref{s3} introduces the structure of generalized irregular square-cell configuration along with its properties. Section \ref{s4} presents the results for the networks introduced in Section \ref{s3}, and Section \ref{s5} concludes by discussing applications and future directions.

\section{Average Distances in Graphs} \label{s1}

Average distance is an essential property in chemical graph theory, which establishes the mean distances in a molecule comprising atoms and bonds. Several works on average distance in graphs were produced, and a review of the literature is presented in this section. Doyle and Graver in 1977 defined the mean distance (average distance) in a graph and produced the bounds based on the diameter of a graph \cite{DoGr77}. The upper bound was investigated for average distance in \cite{Pl84} through the order of a graph, and the same parameter was investigated for 2-connected and 2-edge connected graphs. The relation between independence number and $\mu(G)$ is given in \cite{Ch88}. The measure of the average distance based on vertex removal is investigated in \cite{Al90}. Average distances were investigated for several graph classes including interval graphs \cite{Da93}, $(n,k)$-star graphs \cite{ChCh98a}, arrangement graphs \cite{ChCh98}, square-cell configurations \cite{GuKlRa00}, colored graph \cite{DaGoSl01}, web graph \cite{BrKuMa00}, weighted graphs \cite{Da12}, Sierpi\'{n}ski gasket \cite{HiSc90}, butterfly and hypertree architectures \cite{KlMaNa16}. Dankelmann extensively studied extremal graph theory to provide sharp upper bounds for the average distances in graphs \cite{Da97}, and Firby and Haviland achieved sharp lower bounds \cite{FiHa97}. The relationship between $\mu(G)$ in a graph was investigated by finding a minimum spanning tree in a graph, and the bound in terms of $\delta$ and girth was produced in \cite{DaEn00, Mu14} for $C_3$ and $C_4$-free graphs. Also, the relationship between $\mu(G)$ and the total number of paths containing an edge with some special properties is derived in \cite{BiGy88}. The paper \cite{DaMuSw09} provides sharp upper bounds on $\mu(G)$ of highly vertex-connected graphs, especially when the connectivity parameter $\kappa$ is an odd integer. In \cite{Da10}, the authors have given the relationship between the average distance based on the $k$-packing number and the $k$-domination number. In 2000, $\mu(G)$ was expressed  in terms of its canonical metric representation. The existing theorems on $\mu(G)$, along with a proof of a conjecture related to independence number, are given in \cite{WuZh14}. The work \cite{XuGyGu15} defines the compactness for clique trees (via average distance), characterizes such trees, and provides a polynomial-time algorithm to construct a minimum average distance clique tree for chordal graphs. Ma et al.\ extended the concept to a minimum $\mu(G)$ in a vertex through the idea of proximity and eccentricity of a graph \cite{MaWuZh12}. For further results relating to average distance and other graph parameters (connectivity, degree, subgraph, independence number), the readers could refer to~\cite{BeOePi02, ChLu02, DaOeWu04, Si09, Ju17, BuPi09}.

\section{Cut Method: A Tool for Calculating Wiener Index and Average Distance} \label{s2}

The cut method is a well-established and effective technique for computing topological indices of large networks. It functions by decomposing the associated graph into smaller fragments and subsequently reconstructing the index of these fragments to determine the properties of the entire structure. The principal advantage of this method lies in its ability to compute the topological index of families of chemical graphs without the need for exhaustive, brute-force calculations of distances between all pairs of vertices. The (standard) cut method is applied to molecular graphs that belong to the class of partial cubes.

For dimension $n \geq 1$, the hypercube $Q_n$ is defined as the Cartesian product of $n$ copies of the complete graph $K_2$. The vertex set of $Q_n$ is equivalently made up of all binary strings of length $n$, where two vertices are adjacent if and only if the associated strings differ by precisely one place. The subgraph $H$ of $G$ is considered to be convex if every shortest path between any two vertices in $H$ lies inside $H$. A subgraph $H$ of a graph $G$ is an isometric subgraph if 
\[
d_H(c, d) = d_G(c, d)
\]
holds for all $c, d \in V(H)$. Every convex subgraph is obviously isometric, the reverse is not necessarily true. Partial cubes are, by definition, isometric subgraphs of hypercubes. In the theory of partial cubes, the Djokovi\'{c}-Winkler relation $\Theta$~\cite{Wi84, Dj73} plays an essential role because the cut method is an edge-cut method that decomposes a chemical network using $\Theta$. This relation is defined for any two edges $a = wx$ and $b = yz$ of a graph $G$ by saying that $a\Theta b$ when the following holds:
\[
d_G(w,y) + d_G(x,z) \neq d_G(w,z) + d_G(x,y).
\]
The transitive closure $\Theta^{\ast}$ of the relation $\Theta$ is an equivalence relation on $E(G)$, the corresponding partition is called the $\Theta^{\ast}$-partition. The standard cut method partitions the edge set of a graph $G$ by a relation into classes $F_1, \ldots, F_k$, which are referred to as cuts, where each graph $ G-F_i$, $1 \leq i \leq k$, has at least two connected components. 

\begin{thm} {\rm \cite{klavzar-1995}}
	Let $F_1, \ldots, F_k$ be the $\Theta$-classes of a partial cube $G$, and let $G_1^i$ and $G_2^i$ be the connected components of $ G-F_i$, $1\le i\le k$. Then  
	\[
	W(G) = \sum_{i=1}^k |V(G_1^i)| \times |V(G_2^i)|.
	\]
\end{thm}

In this paper, the standard cut method is enough for our purposes, but we point out that the method has been extended to apply to broader classes of graphs beyond partial cubes. The first article to propose a general cut method (for arbitrary graphs) for the Wiener index was~\cite{Kl06}, see the survey paper~\cite{KlNa15} for the developments till 2015. The extended cut method has been later further elaborated, see, for instance,~\cite{brezovnik-2021, brezovnik-2023, tratnik-2017, tratnik-2020}. Let us also note that very recently  an extended cut method was elaborated for the computation of the Wiener index on hypergraphs in~\cite{romih-2025}. 
    
\section{Generalized Irregular Square-Cell Configuration}
\label{s3}

In this section, we introduce irregular square cell-configurations. They are six-sided convex polygons drawn on the infinite square grid. An irregular square-cell configuration will be denoted by $\mathrm{ISC}(p, q, m, n)$. It is a finite, subgraph (connected) of the infinite square lattice embedded in the plane such that the following two conditions hold. 

\begin{enumerate}
	\item Every inner face of $G$ is a 4-cycle (that is, each bounded face is formed by a cycle of four edges).
	\item The shape of $\mathrm{ISC}(p, q, m, n)$ is determined by the parameters $p$, $q$, $m$, and $n$ in the way as presented in Figure~\ref{fig:ISC}. 
\end{enumerate}

\begin{figure}[ht!]
	\centering 	
	\subfloat[]{\includegraphics[scale=.5]{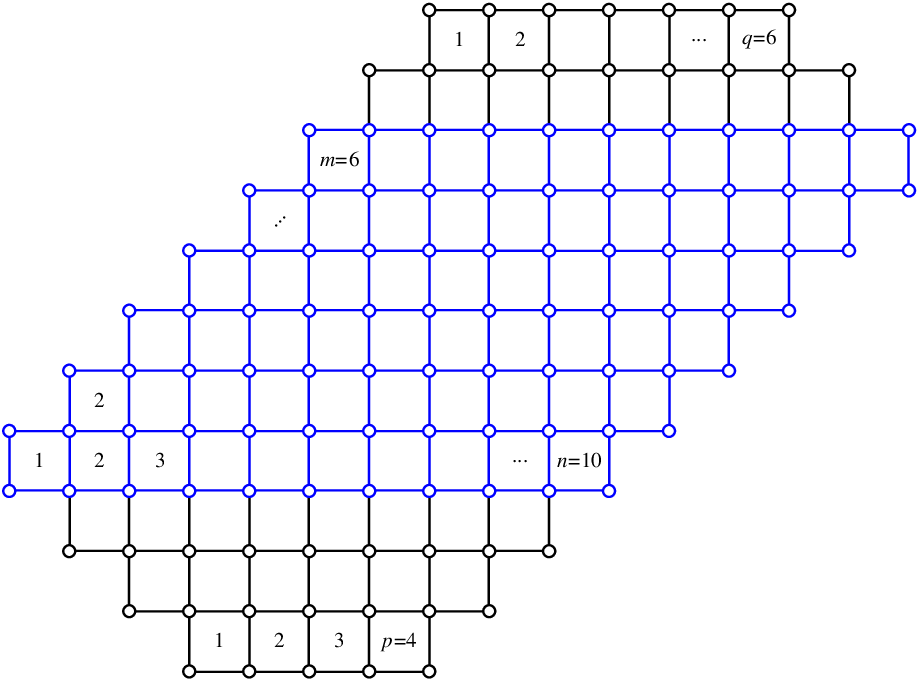}}
	\quad \quad 
	\subfloat[]{\includegraphics[scale=0.5]{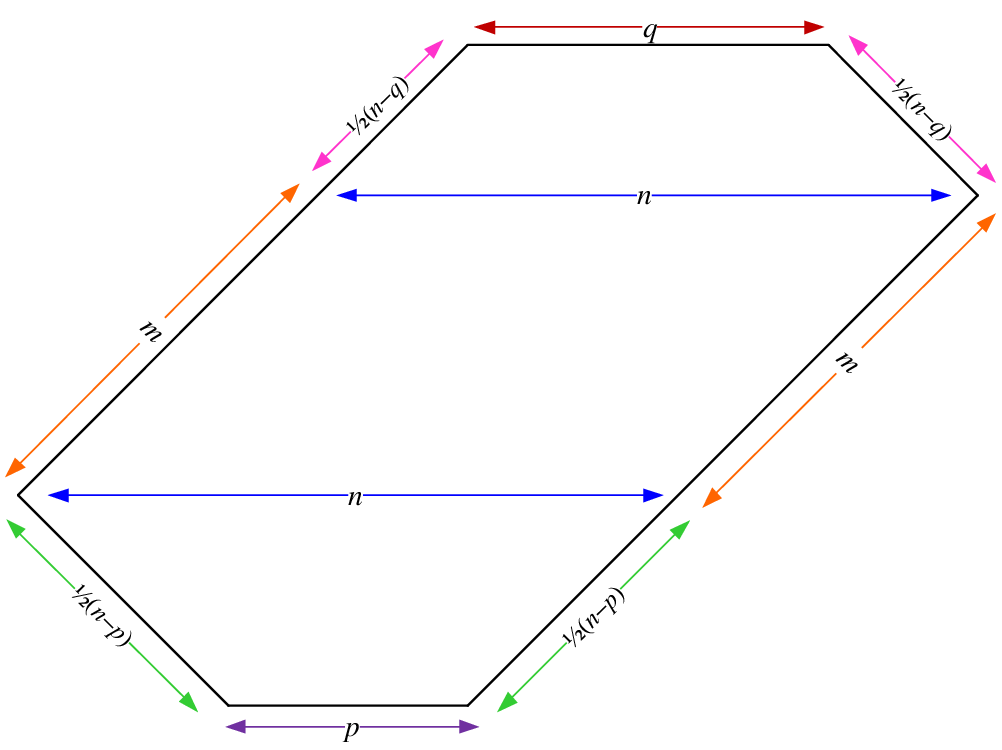}}
	\caption{(a) $\mathrm{ISC}(4, 6, 6, 10)$; (b) $\mathrm{ISC}(p, q, m, n)$}
	\label{fig:ISC}
\end{figure}

That is, in an irregular square-cell configuration $\mathrm{ISC}(p, q, m, n)$, the length of the lower trapezium and the upper trapezium are respectively denoted by $p$ and $q$, the height and the length of the parallelogram are respectively denoted by $m$ and $n$. It consists of 
$\dfrac{1}{4}(2n^2 - p^2 - q^2 + 4mn + 8m + 4n)$ vertices and $\dfrac{1}{2}(2n^2 - p^2 - q^2 + 4mn + 4m + p + q - 2)$
edges. With respect to the parameters $p$, $q$, $m$, and $n$, we need to consider the following three cases:

\begin{enumerate}
	\item $p \leq q-2m+2$,
	\item $p\leq 2m-q-2$,
	\item $p > q-2m+2 $.
\end{enumerate} 

When $q=p$, $m=1$, and $n=3p-2$, the structure $\mathrm{ISC}(p,q,m,n)$ is isomorphic to the hexagonal square-cell configuration. When $q=n$ and $m=1$, the structure $\mathrm{ISC}(p,q,m,n)$ is isomorphic to the trapezium square-cell configuration. When $m=1$, the structure $\mathrm{ISC}(p,q,m,n)$ is isomorphic to the bitrapezium square-cell configurations. The square-cell configurations as shown in Figure~\ref{fig:old-cases} were considered earlier in~\cite{GuKlRa00}. 

\begin{figure}[ht!]
	\centering 	
	\subfloat[]{\includegraphics[scale=0.6]{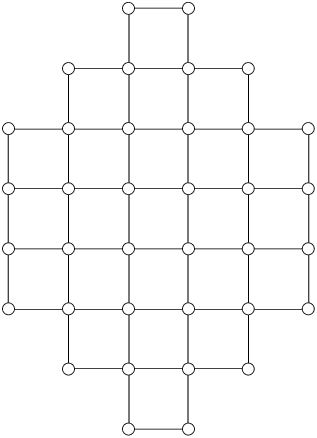}}
	\quad \quad 
	\subfloat[]{\includegraphics[scale=0.6]{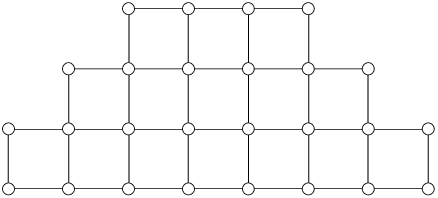}}
	\quad \quad 
	\subfloat[]{\includegraphics[scale=0.6]{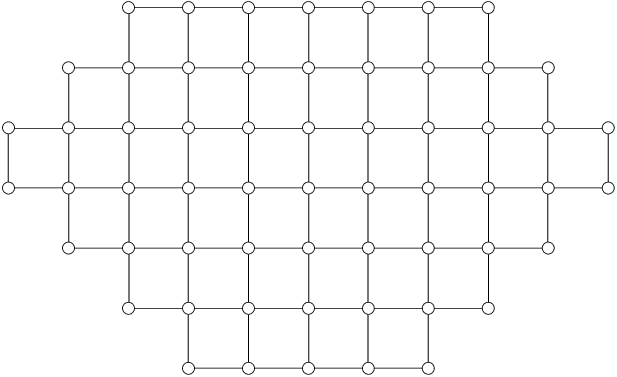}}
	\caption{ (a) Hexagonal square-cell configuration $H(p)$; (b) Trapezium square-cell configuration $T(n, p)$; (c) Bitrapezium square-cell configuration $BT(n, p, q)$.}
	\label{fig:old-cases}
\end{figure}

\section{Main Results} \label{s4}

The number of vertices in the components of horizontal and vertical cuts of all three cases is given in Tables~\ref{t1}, \ref{t2}, \ref{t3}, and \ref{t4}.

\begin{table}[ht!]
	\centering
	\caption{Number of vertices in the horizontal cuts}
			
	\vspace{0.6 cm}\label{t1}
	\scalebox{1}{
		\begin{tabular}{|l|l|}
			\hline
			\multicolumn{1}{|c|}{\textbf{Horizontal Cuts}}                   & \multicolumn{1}{c|}{\textbf{Number of Vertices}}  \\                                            \hline
			\begin{tabular}[c]{@{}l@{}}	$H_{1k}$: \\$1\leq k\leq \frac{n-p}{2}$ 
			\end{tabular}     & \begin{tabular}[c]{@{}l@{}}$f_{1k}= pk+k^2$\\ $f_{2k}=|V|-f_{1k}$  \end{tabular}                           \\ \hline
			\begin{tabular}[c]{@{}l@{}}	$H_{2k}$:  \\ $1\leq   k\leq m$    \\ \end{tabular}     & \begin{tabular}[c]{@{}l@{}}$f_{3k}$ = $  \frac{1}{4}(n^2-p^2+4nk+8k-4)$\\ $f_{4k}=   |V|-f_{3k}$ \end{tabular}                       \\ \hline
			\begin{tabular}[c]{@{}l@{}}	$H_{3k}$: \\ $1\leq k\leq \frac{n-q}{2} $\\ \end{tabular}       & \begin{tabular}[c]{@{}l@{}}$f_{5k}$ =$ \frac{1}{4}(n^2-p^2+4mn+8m+4nk+8k-4k^2-4)$\\ $f_{6k}= |V|-f_{5k}$        
			\end{tabular} 
			\\ \hline
		\end{tabular}     
			}
\end{table}

\begin{thm}
For $m \geq 1$, $p \leq q \leq n $, and $p \leq q-2m+2 $, \\ $W(\mathrm{ISC}(p, q, m, n))$ = $\frac{1}{960}\Big(
			-32 m^5 + 80 m^4 q + 160 m^4 + 320 m^3 n^2 + 1280 m^3 n 
			- 80 m^3 p^2 - 80 m^3 q^2 - 320 m^3 q + 1120 m^3 
			+ 480 m^2 n^3 + 1920 m^2 n^2 - 240 m^2 n p^2 - 240 m^2 n q^2 
			+ 1920 m^2 n + 120 m^2 p^2 q - 240 m^2 p^2 + 40 m^2 q^3 
			- 240 m^2 q^2 + 240 m^2 q - 160 m^2 
			+ 280 m n^4 + 1280 m n^3 - 240 m n^2 p^2 - 240 m n^2 q^2 
			+ 1680 m n^2 + 80 m n p^3 - 480 m n p^2 - 80 m n p 
			+ 80 m n q^3 - 480 m n q^2 - 80 m n q + 320 m n 
			- 10 m p^4 + 160 m p^3 + 60 m p^2 q^2 - 240 m p^2 q 
			- 120 m p^2 - 160 m p - 10 m q^4 + 80 m q^3 - 120 m q^2 
			- 128 m 
			+ 56 n^5 + 280 n^4 - 80 n^3 p^2 - 80 n^3 q^2 + 400 n^3 
			+ 40 n^2 p^3 - 240 n^2 p^2 - 40 n^2 p + 40 n^2 q^3 
			- 240 n^2 q^2 - 40 n^2 q + 80 n^2 
			- 10 n p^4 + 80 n p^3 + 60 n p^2 q^2 - 120 n p^2 - 80 n p 
			- 10 n q^4 + 80 n q^3 - 120 n q^2 - 80 n q - 96 n 
			+ 4 p^5 + 5 p^4 q - 10 p^4 - 20 p^3 q^2 - 20 p^3 
			- 10 p^2 q^3 + 60 p^2 q^2 + 80 p^2 q + 40 p^2 
			+ 20 p q^2 + 16 p 
			+ 5 q^5 - 10 q^4 + 40 q^2 - 80 q\Big)$.
\end{thm}

\begin{proof}
When $m \geq 1$, $p \leq q \leq n $ and $p \leq q-2m+2$, the Wiener index of $\mathrm{ISC}(p, q, m, n)$ can be expressed as follows: 
\begin{align*}
		W(\mathrm{ISC}(p, q, m, n))&= \sum_{k=1}^{\frac{n-p}{2}}
		
		f_{1k}f_{2k} +\sum_{k=1}^{m}
		
		f_{3k}f_{4k}+\sum_{k=1}^{\frac{n-q}{2}}
		
		f_{5k}f_{6k}+\sum_{k=1}^{\frac{2m+n-q-2}{2}}
		
		f_{7k}f_{8k}+\sum_{k=1}^{\frac{q-p-2m+2}{2}}
		
		f_{9k}f_{10k} \\ & \quad +\sum_{k=1}^{p}f_{11k}f_{12k} +\sum_{k=1}^{\frac{q-p+2m-2}{2}}f_{13k}f_{14k} +\sum_{k=1}^{\frac{n-q}{2}}f_{15k}f_{16k}.&
	\end{align*}
	By consulting Figures \ref{fig:5} and \ref{fig:2} and after performing some tedious computations, the formula is derived.	
\end{proof}

\begin{table}[H]
	\centering
	\caption{Number of vertices in the component values for $p \leq q-2m+2 $}\vspace{0.6cm}\label{t2}
	\scalebox{0.8}{
		\begin{tabular}{|l|l|}
			\hline
			\multicolumn{1}{|c|}{\textbf{Vertical Cuts}}                   & \multicolumn{1}{c|}{\textbf{Number of vertices}}  \\                                            \hline
			
			\begin{tabular}[c]{@{}l@{}}	$V_{1k}$:\\ $1\leq k\leq \frac{2m+n-q-2}{2}$  \\ \end{tabular}      & \begin{tabular}[c]{@{}l@{}}$f_{7k} = k^2+k$\\ $f_{8k}= |V|-f_{7k}$  \end{tabular} \\ \hline
			\begin{tabular}[c]{@{}l@{}}	$V_{2k}$:\\ $1\leq k\leq \frac{q-p-2m+2}{2} $ \\ \end{tabular}      & \begin{tabular}[c]{@{}l@{}}$f_{9k} = \frac {1}{4}(4m^2+n^2+q^2+4mn-4mq-4m-2nq-2n+2q+8mk+4nk-4qk+2k^2-2k)$\\ $f_{10k}= |V|-f_{9k}$ \end{tabular}                        \\ \hline
			\begin{tabular}[c]{@{}l@{}}	$V_{3k}$:\\ $1\leq k\leq p $\end{tabular} & \begin{tabular}[c]{@{}l@{}}$f_{11k} = \frac {1}{8}(2n^2-4m^2+p^2-q^2-4mp+4mq-4np+2pq-2p-2q+4m$\\$+4n+8k+8mk+8nk-4pk-4qk)$\\ $f_{12k}= |V|-f_{11k}$       \end{tabular}     \\ \hline
			\begin{tabular}[c]{@{}l@{}}	$V_{4k}$:\\
				$1\leq k\leq \frac{q-p+2m-2}{2} $ \end{tabular} & \begin{tabular}[c]{@{}l@{}}$f_{13k}$ = $\frac {1}{8}(2n^2-4m^2-3p^2-q^2+4mp+4np+4mq-2pq+6p-2q+4m+4n+8mk+8nk$\\$-4pk-4qk+12k-4k^2)$\\ $f_{14k}= |V|-f_{13k}$ \end{tabular}     \\ \hline
			\begin{tabular}[c]{@{}l@{}} $V_{5k}$: \\ $1\leq k\leq \frac{n-q}{2}$\end{tabular}        & \begin{tabular}[c]{@{}l@{}}$f_{15k} = \frac {1}{4}(8m-2n+6q+4mn+2nq+n^2-p^2-2q^2-8+4nk-4qk+12k-4k^2)$\\ $f_{16k}= |V|-f_{15k}$ \\            \end{tabular}               \\ \hline
	\end{tabular}}
\end{table}
		
\begin{figure}[ht!]
			\centering 	
			\subfloat[]{\includegraphics[scale=0.5]{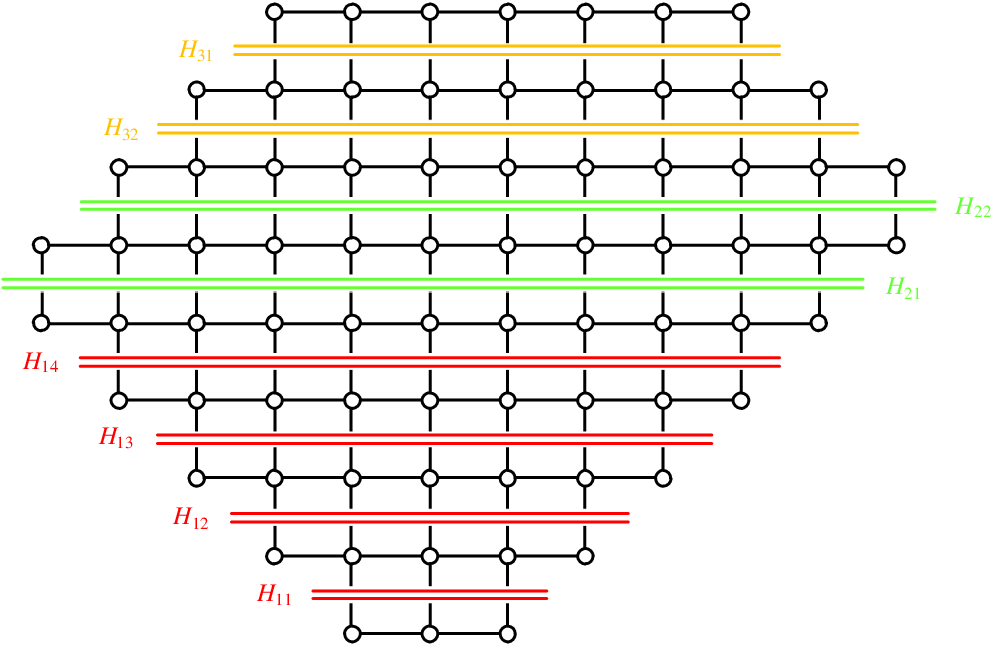}}
			\quad \quad \quad \quad
			\subfloat[]{\includegraphics[scale=0.5]{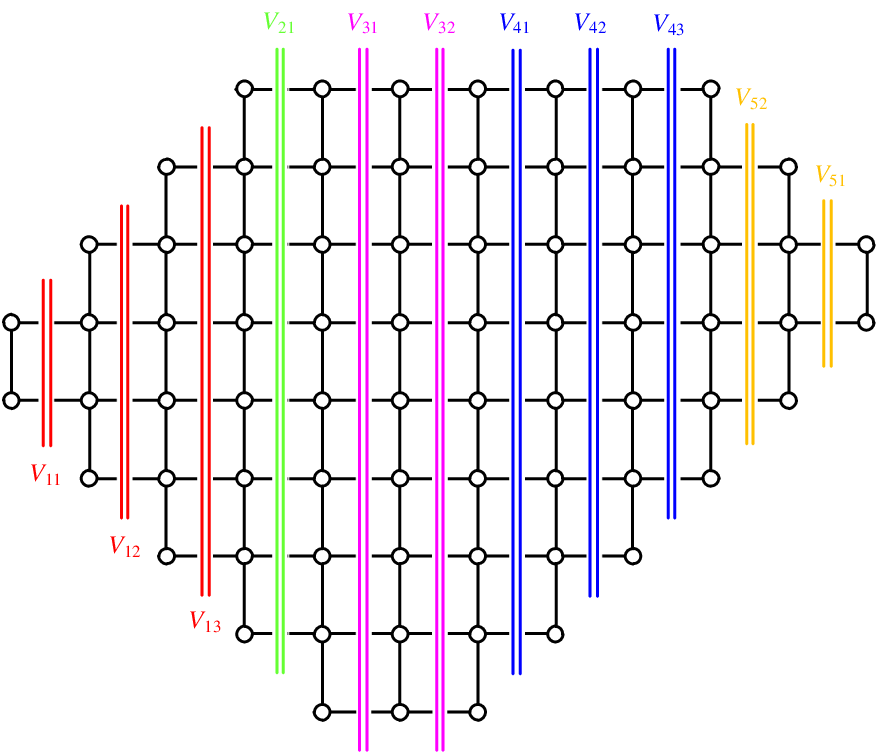}}
			\caption{$p \leq q-2m+2 $ 
				(a) Horizontal cuts; (b) Vertical cuts }
			\label{fig:5}
\end{figure}

\begin{figure}[ht!]
	\centering 	
	\subfloat[]{\includegraphics[scale=.5]{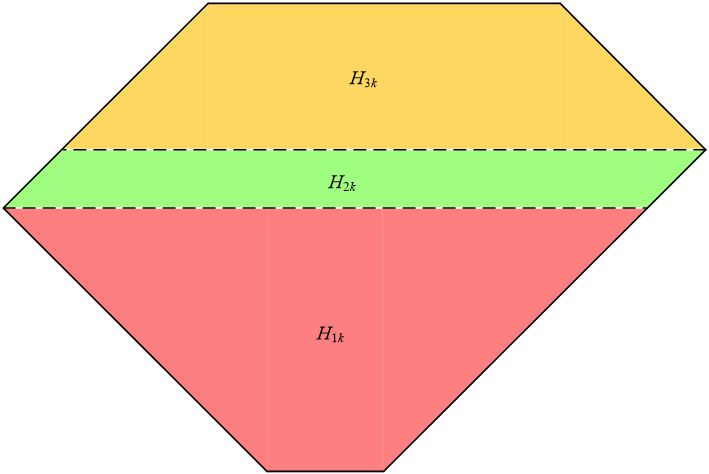}}
	\quad \quad 
	\subfloat[]{\includegraphics[scale=0.5]{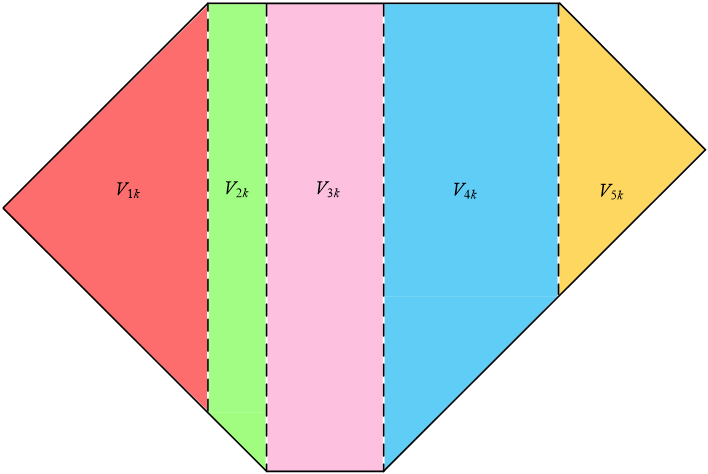}}
	\caption{$p \leq q-2m+2 $ (a) $H_{1k}$: $1\leq k\leq \frac{n-p}{2}, $ $H_{2k}$: $1\leq   k\leq m$, $H_{3k}$: $1\leq k\leq \frac{n-q}{2}$; (b) $V_{1k}$: $1\leq k\leq \frac{2m+n-q-2}{2}$, $V_{2k}$: $1\leq k\leq \frac{q-p-2m+2}{2} $, $V_{3k}$: $1\leq k\leq p $, $V_{4k}$:
		$1\leq k\leq \frac{q-p+2m-2}{2}$,  $V_{5k}$:  $1\leq k\leq \frac{n-q}{2}$.}
	\label{fig:2}
\end{figure}
	
\begin{cor}{\rm\cite{GuKlRa00}}
If $p \geq 2$, then 
$W(H(p))$ =
$\frac{1}{15}(158p^5 - 35p^3 - 3p)$.	
\end{cor}

\begin{cor}{\rm\cite{GuKlRa00}}
	If $n, p \geq 1$, then $W(T(n,p)) 
	= \tfrac{1}{960}\Bigl(
	11n^{5} + 220n^{4} - 30n^{3}p^{2} + 1400n^{3} 
	+ n^{2}(20p^{3} - 360p^{2} - 20p + 3440) 
	+ n(-5p^{4}  + 160p^{3} - 880p^{2} - 160p + 3344) 
	+ 4p^{5} - 20p^{4} + 140p^{3} - 400p^{2} - 144p + 960
	\Bigr).$
\end{cor}

\begin{cor}
   If $n, p, q \geq 1$, then $W(BT(n,p,q)) = \frac{1}{960}\Bigl(
56n^{5} + 560n^{4} - 80n^{3}p^{2} - 80n^{3}q^{2} + 2160n^{3} 
+ 40n^{2}p^{3} - 480n^{2}p^{2} - 40n^{2}p + 40n^{2}q^{3} - 480n^{2}q^{2} - 40n^{2}q + 4000n^{2} 
- 10n p^{4} + 160n p^{3} + 60n p^{2} q^{2} - 840n p^{2} - 160n p 
- 10n q^{4} + 160n q^{3} - 840n q^{2} - 160n q + 3424n 
+ 4p^{5} + 5p^{4}q - 20p^{4} - 20p^{3}q^{2} + 140p^{3} - 10p^{2}q^{3} 
+ 120p^{2}q^{2} - 40p^{2}q - 400p^{2} + 20pq^{2} - 144p 
+ 5q^{5} - 20q^{4} + 120q^{3} - 400q^{2} - 80q + 960
\Bigr)$.
\end{cor}

\begin{thm}
If $m \geq 1$, $p \leq q \leq n$, and $p \leq 2m-q-2 $,  then \\
$W(\mathrm{ISC}(p, q, m, n))$= $\frac{1}{480}\Bigl(160m^3n^2 + 640m^3n + 640m^3 + 240m^2n^3 + 960m^2n^2 - 120m^2np^2 - 120m^2nq^2 + 960m^2n - 240m^2p^2 - 240m^2q^2 + 140mn^4 + 640mn^3 - 120mn^2p^2 - 120mn^2q^2 + 840mn^2 + 40mnp^3 - 240mnp^2 - 40mnp + 40mnq^3 - 240mnq^2 - 40mnq + 160mn + 80mp^3 + 60mp^2q^2 - 80mp + 80mq^3 - 80mq - 160m + 28n^5 + 140n^4 - 40n^3p^2 - 40n^3q^2 + 200n^3 + 20n^2p^3 - 120n^2p^2 - 20n^2p + 20n^2q^3 - 120n^2q^2 - 20n^2q + 40n^2 - 5np^4 + 40np^3 + 30np^2q^2 - 60np^2 - 40np - 5nq^4 + 40nq^3 - 60nq^2 - 40nq - 48n + 2p^5 - 10p^4 - 10p^3q^2 - 10p^3 - 10p^2q^3 + 10p^2q + 40p^2 + 10pq^2 + 8p + 2q^5 - 10q^4 - 10q^3 + 40q^2 + 8q \Bigr)$.	
\end{thm}

\begin{proof}
When $m \geq 1$, $p \leq q \leq n$, and $p \leq 2m-q-2$, the Wiener index of $\mathrm{ISC}(p, q, m, n)$ can be expressed as follows: 
		\begin{align*}
		
		W(\mathrm{ISC}(p, q, m, n)) &= \sum_{k=1}^{\frac{n-p}{2}}
		
		f_{1k}f_{2k}+\sum_{k=1}^{m}
		
		f_{3k}f_{4k}+\sum_{k=1}^{\frac{n-q}{2}}
		
	f_{5k}f_{6k}+\sum_{k=1}^{\frac{n-p}{2}}
		
		f_{7k}f_{8k}+\sum_{k=1}^{p}
		
		f_{9k}f_{10k} \\ &\quad +\sum_{k=1}^{\frac{2m-p-q-2}{2}}f_{11k}f_{12k} +\sum_{k=1}^{q}f_{13k}f_{14k} +\sum_{k=1}^{\frac{n-q}{2}}f_{15k}f_{16k}.&
		\end{align*}
By consulting Figures \ref{fig:6} and \ref{fig:3}, and after performing some tedious computations, the asserted formula is derived.	
\end{proof}	

\begin{figure}[ht!]
	\centering 	
	\subfloat[]{\includegraphics[scale=0.5]{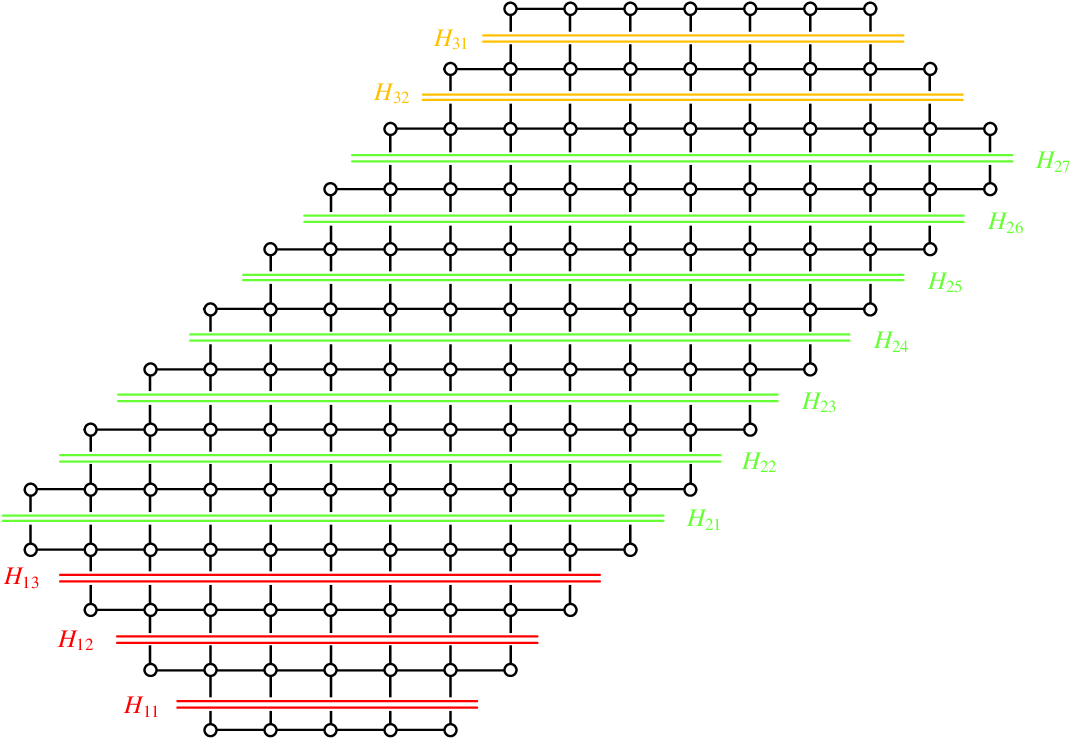}}
	\quad \quad 
	\subfloat[]{\includegraphics[scale=0.5]{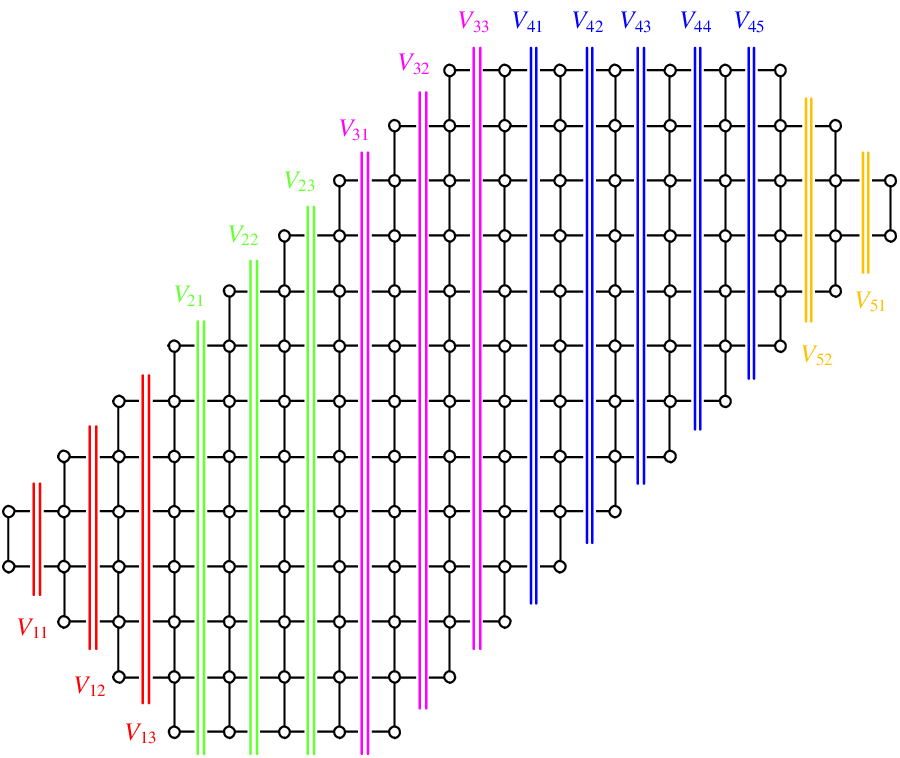}}
	\caption{$p \leq 2m-q-2 $ (a) Horizontal cuts; (b) Vertical cuts.}
	\label{fig:6}
\end{figure}

\begin{figure}[ht!]
	\centering 	
	\subfloat[]{\includegraphics[scale=0.5]{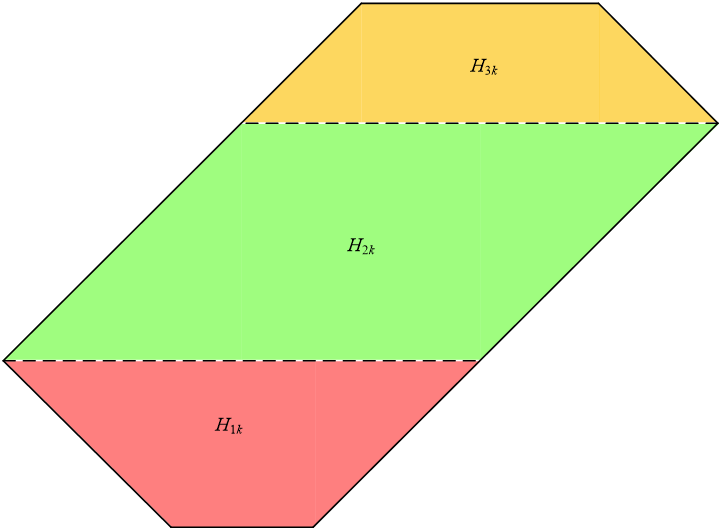}}
	\quad \quad 
	\subfloat[]{\includegraphics[scale=0.5]{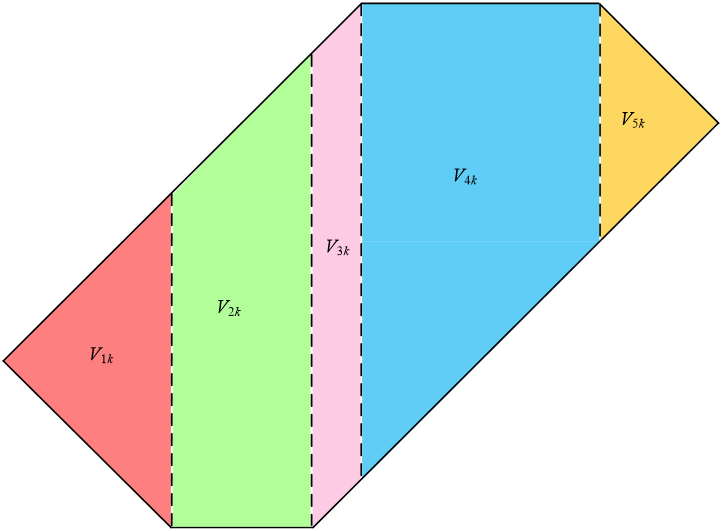}}
	\caption{$p\leq 2m-q-2$ (a) $H_{1k}$: $1\leq k\leq \frac{n-p}{2}, $ $H_{2k}$: $1\leq   k\leq m$, $H_{3k}$: $1\leq k\leq \frac{n-q}{2}$; (b) $V_{1k}$: $1\leq k\leq \frac{n-p}{2}$,  $V_{2k}$: $1\leq k\leq p $, $V_{3k}$: $1\leq k\leq \frac{1}{2}(2m-p-q-2)$, $V_{4k}$: $1\leq k\leq q $, $V_{5k}$: $1\leq k\leq \frac{n-q}{2}$.}
	\label{fig:3}
\end{figure}

\begin{table}[H]
	\centering
	\caption{Number of vertices in the  component values for $p\leq 2m-q-2$}\vspace{0.6cm}\label{t3}
	\scalebox{0.85}{
		\begin{tabular}{|l|l|}
			\hline
			\multicolumn{1}{|c|}{\textbf{Vertical Cuts}}                   & \multicolumn{1}{c|}{\textbf{Number of vertices}}  \\                                                \hline
			
			\begin{tabular}[c]{@{}l@{}}	$V_{1k}$:\\ $1\leq k\leq \frac{n-p}{2}$  \\ \end{tabular}      & \begin{tabular}[c]{@{}l@{}}$f_{7k} = k^2+k$\\ $f_{8k}= |V|-f_{7k}$                 
				\end{tabular} \\ \hline
			\begin{tabular}[c]{@{}l@{}}	$V_{2k}$:\\ $1\leq k\leq p $ \\ \end{tabular}      & \begin{tabular}[c]{@{}l@{}}$f_{9k}$ = $\frac {1}{4}(n^2+p^2-2np+2n-2p+4nk-4pk+2k^2+6k)$\\ $f_{10k}= |V|-f_{9k}$  \end{tabular}                        \\ \hline
			\begin{tabular}[c]{@{}l@{}}	$V_{3k}$:\\ $1\leq k\leq \frac{1}{2}(2m-p-q-2)$\end{tabular} & \begin{tabular}[c]{@{}l@{}}$f_{11k} = \frac {1}{4}(n^2-p^2+2np+2n+4p+4nk+8k)$\\ $f_{12k}= |V|-f_{11k}$ \end{tabular}     \\ \hline
			\begin{tabular}[c]{@{}l@{}}	$V_{4k}$:\\
				$1\leq k\leq q $ \end{tabular} & \begin{tabular}[c]{@{}l@{}}$f_{13k}$ = $\frac {1}{4}(n^2-p^2+4mn-2nq+8m-2n-4q+4nk-2k^2+10k-8)$\\ $f_{14k}= |V|-f_{13k}$ \\     \end{tabular}     \\ \hline
			\begin{tabular}[c]{@{}l@{}}	$V_{5k}$:\\ $1\leq k\leq \frac{n-q}{2}$  \end{tabular} & \begin{tabular}[c]{@{}l@{}}$f_{15k}$ = $ \frac {1}{4}(n^2-p^2+4mn+2nq+8m-2n+6q-2q^2+4nk-4qk-4k^2+12k-8)$\\ $f_{16k}= |V|-f_{15k}$ \end{tabular}     \\ \hline
	\end{tabular}}
\end{table}
						
\begin{thm}
If $m \geq 1$, $p \leq q \leq n$, and $p > q-2m+2$, then \\
	$W(\mathrm{ISC}(p, q, m, n))= \frac{1}{1920}\Big(-32m^5 
	+ 80m^4p + 80m^4q + 160m^4 
	+ 640m^3n^2 + 2560m^3n - 80m^3p^2 - 160m^3pq - 320m^3p - 80m^3q^2 - 320m^3q + 2400m^3 
	+ 960m^2n^3 + 3840m^2n^2 - 480m^2np^2 - 480m^2nq^2 + 3840m^2n + 40m^2p^3 + 120m^2p^2q - 720m^2p^2 + 120m^2pq^2 + 480m^2pq + 240m^2p + 40m^2q^3 - 720m^2q^2 + 240m^2q - 160m^2 
	+ 560mn^4 + 2560mn^3 - 480mn^2p^2 - 480mn^2q^2 + 3360mn^2 + 160mnp^3 - 960mnp^2 - 160mnp + 160mnq^3 - 960mnq^2 - 160mnq + 640mn 
	- 10mp^4 - 40mp^3q + 240mp^3 + 180mp^2q^2 - 240mp^2q - 120mp^2 - 40mpq^3 - 240mpq^2 - 240mpq - 160mp - 10mq^4 + 240mq^3 - 120mq^2 - 160mq - 448m 
	+ 112n^5 + 560n^4 - 160n^3p^2 - 160n^3q^2 + 800n^3 + 80n^2p^3 - 480n^2p^2 - 80n^2p + 80n^2q^3 - 480n^2q^2 - 80n^2q + 160n^2 
	- 20np^4 + 160np^3 + 120np^2q^2 - 240np^2 - 160np - 20nq^4 + 160nq^3 - 240nq^2 - 160nq - 192n 
	+ 9p^5 + 5p^4q - 30p^4 - 30p^3q^2 + 40p^3q - 20p^3 - 30p^2q^3 + 60p^2q^2 + 100p^2q + 120p^2 
	+ 5pq^4 + 40pq^3 + 100pq^2 - 80pq - 64p 
	+ 9q^5 - 30q^4 - 20q^3 + 120q^2 - 64q \Big)$.
\end{thm}

	\begin{proof}
When $m \geq 1$, $p \leq q \leq n$, and $p > q-2m+2$, the Wiener index of $\mathrm{ISC}(p, q, m, n)$ can be expressed as follows: 
		\begin{align*}
	
		W(\mathrm{ISC}(p, q, m, n)) &= \sum_{k=1}^{\frac{n-p}{2}}
		
		f_{1k}f_{2k} +\sum_{k=1}^{m}
		
		f_{3k}f_{4k}+\sum_{k=1}^{\frac{n-q}{2}}
		
		f_{5k}f_{6k}+\sum_{k=1}^{\frac{n-p}{2}}
		
		f_{7k}f_{8k}+\sum_{k=1}^{\frac{2m+p-q-2}{2}}
		
		f_{9k}f_{10k}+ \\ & \quad \sum_{k=1}^{\frac{p+q-2m+2}{2}}f_{11k}f_{12k}  +\sum_{k=1}^{2m-p+q-2}f_{13k}f_{14k} +\sum_{k=1}^{\frac{n-q}{2}}f_{15k}f_{16k}.&
		\end{align*}
Using Figures \ref{fig:7} and \ref{fig:4} and after some tedious computation, the asserted formula is derived.	
\end{proof}
		
	\begin{figure}[ht!]
		\centering 	
		\subfloat[]{\includegraphics[scale=0.5]{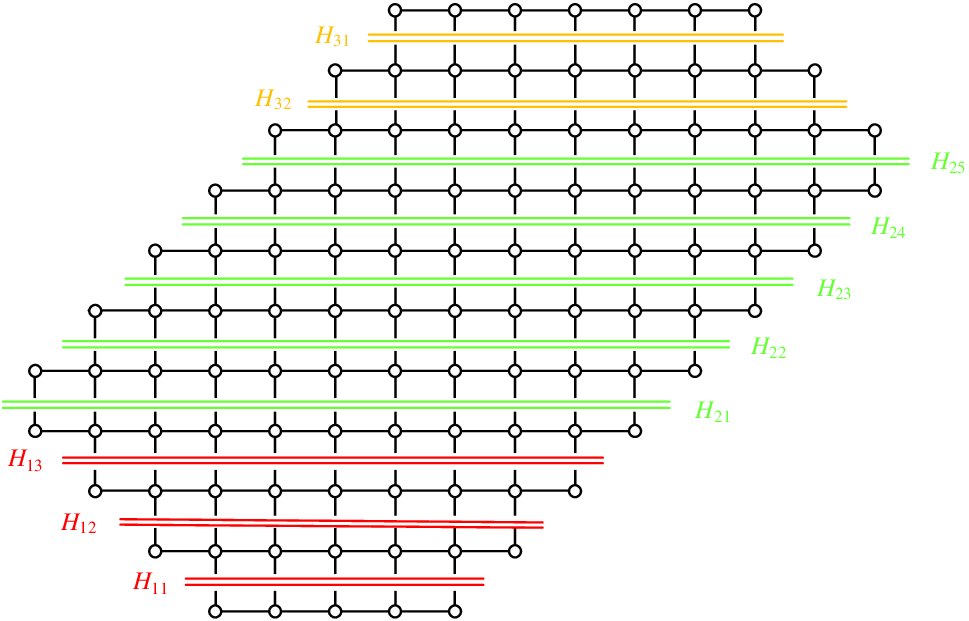}}
		\quad \quad \quad \quad
		\subfloat[]{\includegraphics[scale=0.5]{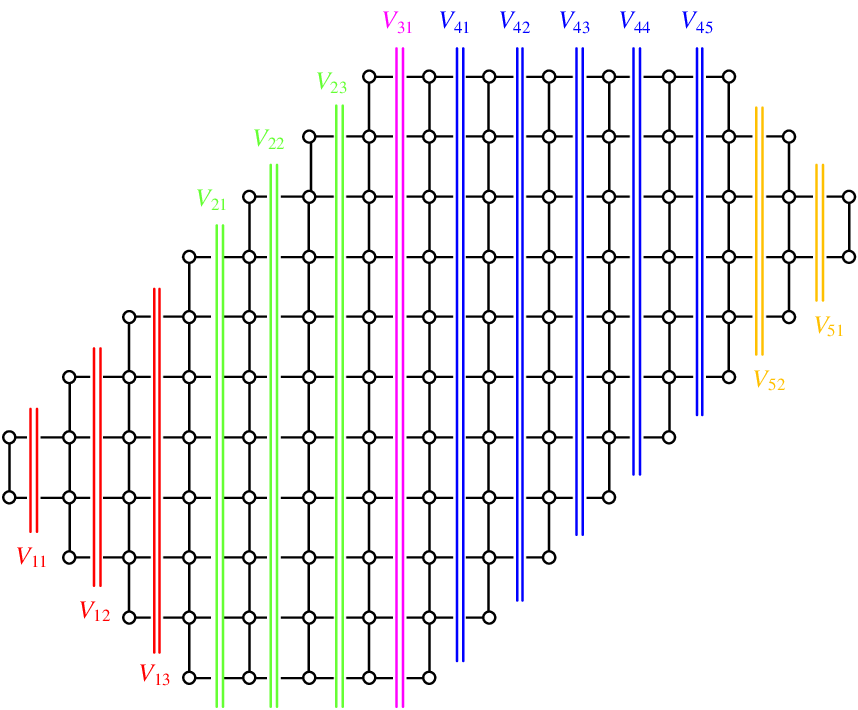}}
		\caption{$p > q-2m+2 $ 
			(a) Horizontal cuts; (b) Vertical cuts }
		\label{fig:7}
	\end{figure}

\begin{figure}[ht!]
	\centering 	
	\subfloat[]{\includegraphics[scale=0.45]{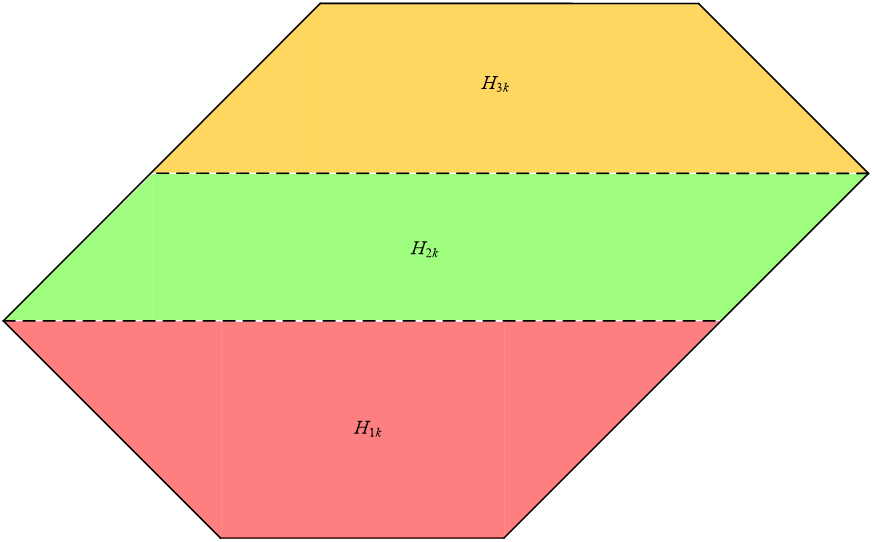}}
	\quad \quad 
	\subfloat[]{\includegraphics[scale=0.45]{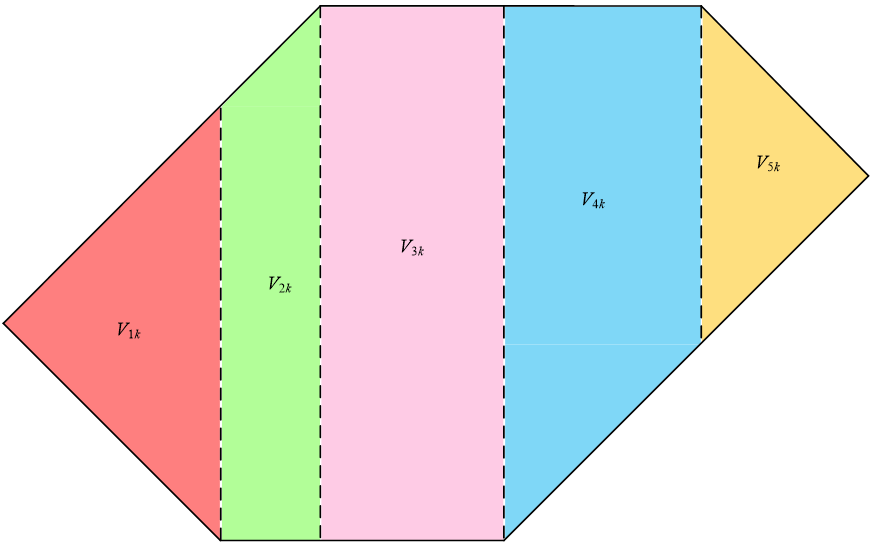}} 
	\caption{$p > q-2m+2 $ (a)  $H_{1k}$: $1\leq k\leq \frac{n-p}{2}$,  $H_{2k}$: $1\leq   k\leq m$, $H_{3k}$: $1\leq k\leq \frac{n-q}{2}$; (b) $V_{1k}$: $1\leq k\leq \frac{n-p}{2}$, $V_{2k}$: $1\leq k\leq \frac{1}{2}(2m+p-q-2)$, $V_{3k}$: $1\leq k\leq \frac{1}{2}(p+q-2m+2)$, $V_{4k}$: $1\leq k\leq \frac{1}{2}(2m-p+q-2)$, $V_{5k}$: $1\leq k\leq \frac{n-q}{2}$.}
	\label{fig:4}
\end{figure}

\begin{table}[ht!]
	\centering
	\caption{Number of vertices in the component values for $p > q-2m+2$.}\vspace{0.6cm}\label{t4}
	\scalebox{0.8}{
		\begin{tabular}{|l|l|}
			\hline
			\multicolumn{1}{|c|}{\textbf{Vertical Cuts}}                   & \multicolumn{1}{c|}{\textbf{Number of vertices}}  \\                                                \hline
			
			\begin{tabular}[c]{@{}l@{}}	$V_{1k}$:\\ $1\leq k\leq \frac{n-p}{2}$  \\ \end{tabular}      & \begin{tabular}[c]{@{}l@{}}$f_{7k} = k^2+k$\\ $f_{8k}= |V|-f_{7k}$  \end{tabular} \\ \hline
			\begin{tabular}[c]{@{}l@{}}	$V_{2k}$:\\ $1\leq k\leq \frac{1}{2}(2m+p-q-2) $ \\ \end{tabular}      & \begin{tabular}[c]{@{}l@{}}$f_{9k}$ = $\frac {1}{4}(n^2+p^2-2np+2n-2p+4nk-4pk+2k^2+6k)$\\ $f_{10k}= |V|-f_{9k}$  \end{tabular}                        \\ \hline
			\begin{tabular}[c]{@{}l@{}}	$V_{3k}$:\\ $1\leq k\leq \frac{1}{2}(p+q-2m+2)$\end{tabular} & \begin{tabular}[c]{@{}l@{}}$f_{11k} = \frac {1}{8}(4m^2+2n^2-p^2+q^2+8mn-4mp-4mq-4nq +2pq+4m-4n+6p -2q $ \\ \quad \quad \quad \quad  $ +8mk+8nk-4pk-4qk+8k-8)$\\ $f_{12k}= |V|-f_{11k}$ \end{tabular}     \\ \hline
			\begin{tabular}[c]{@{}l@{}}	$V_{4k}$:\\
				$1\leq k\leq \frac{1}{2}(2m-p+q-2) $ \end{tabular} & \begin{tabular}[c]{@{}l@{}}$f_{13k}$ = $\frac {1}{8}(2n^2-4m^2-3p^2-q^2+4mp+4mq-2pq+4np+4m +4n+6p-2q+8mk $ \\ \quad \quad \quad \quad $+8nk-4pk-4qk-4k^2+12k)$\\ $f_{14k}= |V|-f_{13k}$ \\      \end{tabular}     \\ \hline
			\begin{tabular}[c]{@{}l@{}}	$V_{5k}$:\\ $1\leq k\leq \frac{n-q}{2}$  \end{tabular} & \begin{tabular}[c]{@{}l@{}}$f_{15k}$ = $ \frac {1}{8}(2n^2-2p^2-4q^2+8mn+4nq+16m-4n+12q+8nk-8qk-8k^2+24k-16)$\\ $f_{16k}= |V|-f_{15k}$ \end{tabular}     \\ \hline
	\end{tabular}}
\end{table}
	
\begin{thm} 
If $m \geq 1$, $p \leq q \leq n$, and $p \leq q-2m+2$, then\\  
$\mu(\mathrm{ISC}(p, q, m, n)) = \frac{1}{30(8m + 4n + 4mn + 2n^{2} - p^{2} - q^{2})(8m + 4n + 4mn + 2n^{2} - p^{2} - q^{2} - 4)}\Big(
-32m^{5} + 80m^{4}q + 160m^{4} + 320m^{3}n^{2} + 1280m^{3}n - 80m^{3}p^{2} - 80m^{3}q^{2} - 320m^{3}q + 1120m^{3} 
+ 480m^{2}n^{3} + 1920m^{2}n^{2} - 240m^{2}n p^{2} - 240m^{2}n q^{2} + 1920m^{2}n + 120m^{2}p^{2}q - 240m^{2}p^{2} 
+ 40m^{2}q^{3} - 240m^{2}q^{2} + 240m^{2}q - 160m^{2} + 280m n^{4} + 1280m n^{3} - 240m n^{2}p^{2} - 240m n^{2}q^{2} 
+ 1680m n^{2} + 80m n p^{3} - 480m n p^{2} - 80m n p + 80m n q^{3} - 480m n q^{2} - 80m n q + 320m n 
- 10m p^{4} + 160m p^{3} + 60m p^{2}q^{2} - 240m p^{2}q - 120m p^{2} - 160m p - 10m q^{4} + 80m q^{3} - 120m q^{2} - 128m 
+ 56n^{5} + 280n^{4} - 80n^{3}p^{2} - 80n^{3}q^{2} + 400n^{3} + 40n^{2}p^{3} - 240n^{2}p^{2} - 40n^{2}p 
+ 40n^{2}q^{3} - 240n^{2}q^{2} - 40n^{2}q + 80n^{2} - 10n p^{4} + 80n p^{3} + 60n p^{2}q^{2} - 120n p^{2} - 80n p 
- 10n q^{4} + 80n q^{3} - 120n q^{2} - 80n q - 96n + 4p^{5} + 5p^{4}q - 10p^{4} - 20p^{3}q^{2} - 20p^{3} 
- 10p^{2}q^{3} + 60p^{2}q^{2} + 80p^{2}q + 40p^{2} + 20p q^{2} + 16p + 5q^{5} - 10q^{4} + 40q^{2} - 80q
\Big). $
\end{thm}

\begin{cor}
For $p \geq 2$, $\mu(H(p))$ = $\frac{158p^4-35p^2-3}{120p^3-30p}$.		
\end{cor}

\begin{cor}
If $n, p \geq 1$, then $\mu(T(n, p))$ = $\frac{1}{30\,(n^{2}+8n-p^{2}+4)\,(n^{2}+8n-p^{2}+8)}
\Bigl(
11n^{5}+220n^{4}-30n^{3}p^{2}+1400n^{3}
+20n^{2}p^{3}-360n^{2}p^{2}-20n^{2}p+3440n^{2}
-5np^{4}+160np^{3}-880np^{2}-160np+3344n
+4p^{5}-20p^{4}+140p^{3}-400p^{2}-144p+960
\Bigr)$.	
\end{cor}

\begin{cor}
If $n, p, q \geq 1$, then $\mu(BT(n,p,q))=
\frac{1}{30\,(2n^{2}+8n-p^{2}-q^{2}+4)\,(2n^{2}+8n-p^{2}-q^{2}+8)}
\Bigl(
56n^{5}+560n^{4}-80n^{3}p^{2}-80n^{3}q^{2}+2160n^{3}
+40n^{2}p^{3}-480n^{2}p^{2}-40n^{2}p
+40n^{2}q^{3}-480n^{2}q^{2}-40n^{2}q+4000n^{2}
-10np^{4}+160np^{3}+60np^{2}q^{2}-840np^{2}-160np
-10nq^{4}+160nq^{3}-840nq^{2}-160nq+3424n
+4p^{5}+5p^{4}q-20p^{4}-20p^{3}q^{2}+140p^{3}
-10p^{2}q^{3}+120p^{2}q^{2}-40p^{2}q-400p^{2}
+20pq^{2}-144p
+5q^{5}-20q^{4}+120q^{3}-400q^{2}-80q+960
\Bigr)$.
\end{cor}
\begin{thm} 
If $m \geq 1$, $p \leq q \leq n$, and $p \leq 2m-q-2$, then\\ 
$\mu(\mathrm{ISC}(p, q, m, n)= \frac{1}{
	15 \left( 8m + 4n + 4mn + 2n^{2} - p^{2} - q^{2} \right)
	\left( 8m + 4n + 4mn + 2n^{2} - p^{2} - q^{2} - 4 \right)
}
\Big(
160m^{3}n^{2} + 640m^{3}n + 640m^{3} + 240m^{2}n^{3} + 960m^{2}n^{2}
-120m^{2}np^{2} - 120m^{2}nq^{2} + 960m^{2}n - 240m^{2}p^{2} - 240m^{2}q^{2}
+ 140mn^{4} + 640mn^{3} - 120mn^{2}p^{2} - 120mn^{2}q^{2} + 840mn^{2}
+ 40mnp^{3} - 240mnp^{2} - 40mnp + 40mnq^{3} - 240mnq^{2} - 40mnq
+ 160mn + 80mp^{3} + 60mp^{2}q^{2} - 80mp + 80mq^{3} - 80mq - 160m
+ 28n^{5} + 140n^{4} - 40n^{3}p^{2} - 40n^{3}q^{2} + 200n^{3} + 20n^{2}p^{3}
-120n^{2}p^{2} - 20n^{2}p + 20n^{2}q^{3} - 120n^{2}q^{2} - 20n^{2}q + 40n^{2}
- 5np^{4} + 40np^{3} + 30np^{2}q^{2} - 60np^{2} - 40np - 5nq^{4} + 40nq^{3}
-60nq^{2} - 40nq - 48n + 2p^{5} - 10p^{4} - 10p^{3}q^{2} - 10p^{3}
-10p^{2}q^{3} + 10p^{2}q + 40p^{2} + 10pq^{2} + 8p + 2q^{5} - 10q^{4}
-10q^{3} + 40q^{2} + 8q
\Big) $.

\end{thm}
\begin{thm} 
If $m \geq 1$, $p \leq q \leq n$, and  $p > q-2m+2$, then \\ 
$ \mu(\mathrm{ISC}(p, q, m, n))=
\frac{1}{60 \left( 8m + 4n + 4mn + 2n^2 - p^2 - q^2 \right)\left( 8m + 4n + 4mn + 2n^2 - p^2 - q^2 - 4 \right)}
\Big(
-32m^5 + 80m^4p + 80m^4q + 160m^4 + 640m^3n^2 + 2560m^3n - 80m^3p^2 - 160m^3pq - 320m^3p - 80m^3q^2 - 320m^3q + 2400m^3 + 960m^2n^3 + 3840m^2n^2 - 480m^2np^2 - 480m^2nq^2 + 3840m^2n + 40m^2p^3 + 120m^2p^2q - 720m^2p^2 + 120m^2pq^2 + 480m^2pq + 240m^2p + 40m^2q^3 - 720m^2q^2 + 240m^2q - 160m^2 + 560mn^4 + 2560mn^3 - 480mn^2p^2 - 480mn^2q^2 + 3360mn^2 + 160mnp^3 - 960mnp^2 - 160mnp + 160mnq^3 - 960mnq^2 - 160mnq + 640mn - 10mp^4 - 40mp^3q + 240mp^3 + 180mp^2q^2 - 240mp^2q - 120mp^2 - 40mpq^3 - 240mpq^2 - 240mpq - 160mp - 10mq^4 + 240mq^3 - 120mq^2 - 160mq - 448m + 112n^5 + 560n^4 - 160n^3p^2 - 160n^3q^2 + 800n^3 + 80n^2p^3 - 480n^2p^2 - 80n^2p + 80n^2q^3 - 480n^2q^2 - 80n^2q + 160n^2 - 20np^4 + 160np^3 + 120np^2q^2 - 240np^2 - 160np - 20nq^4 + 160nq^3 - 240nq^2 - 160nq - 192n + 9p^5 + 5p^4q - 30p^4 - 30p^3q^2 + 40p^3q - 20p^3 - 30p^2q^3 + 60p^2q^2 + 100p^2q + 120p^2 + 5pq^4 + 40pq^3 + 100pq^2 - 80pq - 64p + 9q^5 - 30q^4 - 20q^3 + 120q^2 - 64q
\Big).
$
\end{thm}
\section{Conclusions and Future Directions}
\label{s5}

In this study, we introduced irregular square cell configurations, an extension to the well-studied class of square-cell configurations, including hexagonal, trapezium and bitrapezium configurations. An essential property of the square-cell configurations is that each inner face constitutes a 4-cycle and also admits structural properties such as disruptions and asymmetries, which is in contrast to their symmetric nature. To create a unified framework that includes both regular and irregular square-cell networks as exceptional cases, we deduced generalized expressions for the $W(G)$ and $\mu(G)$. Our results demonstrate how irregularity affects the topological behavior of such networks, changes distance distributions, and affects the growth of pairwise distances. The study of the Wiener index and average distance in irregular square-cell configuration has important implications for many areas of science and engineering. These indices provide a quantitative framework for understanding the boiling temperatures, reactivity, and molecular stability of irregular lattice-like substances in chemical graph theory, such as polymers and benzenoid hydrocarbons. Similarly, in field of materials science, the irregular square-cell configuration models exactly capture the geometric distortions and defects present in real crystal lattices, that enables the researchers to predict how such imperfections affect physical properties such as  strength, conductivity, and resilience from the perspective of transportation and communication networks, Wiener index and average distance measure efficiency, robustness, and fault tolerance, especially in urban grid systems or wireless sensor deployments where irregularities naturally arise due to obstacles, varying densities, or incomplete connectivity. 

In theoretical computer science, these measures facilitate the design and analysis of interconnection networks, resource allocation, and optimizing routing algorithms in irregular but planar networks. By extending these concepts beyond regular grids, researchers attains more accurate layout for analyzing, modeling, and optimizing diverse systems in engineering, network science, chemistry, and physical sciences, where irregularity is the rule rather than the exception. This idea not only enhances the theoretical understanding of distance-based graph invariants but also provides noteworthy applications in modeling real-world networks with imperfect or heterogeneous geometries.

\section*{Acknowledgements}
Sandi Klav\v{z}ar was supported under the grants P1-0297, N1-0285, N1-0355 of the Slovenian Research and Innovation Agency.\\

\end{document}